\documentclass[12pt,a4paper,table]{amsart}

\usepackage[T1]{fontenc}
\usepackage[utf8]{inputenc}

\usepackage{amsmath,amssymb,amsfonts,amsthm}
\usepackage{mathtools}
\usepackage{newtxtext} 
\usepackage{newtxmath} 
\usepackage[margin=3cm]{geometry}
\usepackage{xspace}
\usepackage{graphicx}
\usepackage{hyperref}
\usepackage{enumitem}
\usepackage{varioref} 
\usepackage{rotating} 
\usepackage{pdflscape} 
\usepackage{afterpage} 

\usepackage{tikz}
\usetikzlibrary{positioning}
\usetikzlibrary{decorations.pathreplacing,calligraphy} 
\usepackage{diagbox}

\usepackage{hyperref}

\hypersetup{
    pdfnewwindow=true,      
    colorlinks=true,       
    linkcolor=blue,          
    citecolor=blue,        
    filecolor=magenta,      
    urlcolor=cyan           
}
\usepackage[sort&compress,comma,numbers]{natbib}

\renewcommand{\newblock}{} 

\DeclareSymbolFont{AMSb}{U}{msb}{m}{n}
\DeclareMathSymbol{\lr}{\mathbin}{AMSb}{'036}
\DeclareMathSymbol{\rr}{\mathbin}{AMSb}{'037}

\newcommand{\card}[1]{\vert #1 \vert}

\newcommand{\klam}[1]{\ensuremath{\left\langle #1 \right\rangle}}
\newcommand{\set}[1]{\ensuremath{\{#1\}}}
\newcommand{\tiff}{if and only if \ }
\newcommand{\Implies}{\Rightarrow}
\newcommand{\Iff}{\Longleftrightarrow}
\newcommand{\tand}{\text{ and }}
\newcommand{\tor}{\text{ or }}
\newcommand\qtand{\quad\text{and}\quad}

\newcommand{\z}{\emptyset}

\newcommand{\poss}[1]{\ensuremath{\klam{#1}}}
\newcommand{\nec}[1]{\ensuremath{[ #1 ]}}
\newcommand{\suff}[1]{\ensuremath{[[ #1 ]]}}
\newcommand{\dusuff}[1]{\ensuremath{\klam{\klam{ #1 }}}}
\newcommand{\necu}[1]{\ensuremath{[ #1 ]]}}
\newcommand{\possu}[1]{\ensuremath{\klam{ #1 }\rangle}}

\newcommand{\nece}[1]{\ensuremath{[ #1 ]_{E}}}
\newcommand{\suffe}[1]{\ensuremath{[[ #1 ]]_{E}}}

\newcommand{\QED}{\hfill$\dashv$}

\renewcommand{\phi}{\varphi}

\renewcommand{\S}{\ensuremath{\mathcal{S}}\xspace}
\newcommand{\I}{\ensuremath{\mathcal{I}}\xspace}
\newcommand{\Act}{\mathtt{{Act}}}

\newcommand{\oX}{\overline{X}}
\newcommand{\oY}{\overline{Y}}
\newcommand{\oZ}{\overline{Z}}

\newcommand{\uX}{\underline{X}}
\newcommand{\uY}{\underline{Y}}
\newcommand{\uZ}{\underline{Z}}

\newcommand{\ovarphi}{\overline{\varphi}}
\newcommand{\oalpha}{\overline{\alpha}}
\newcommand{\uvarphi}{\underline{\varphi}}
\newcommand{\ualpha}{\underline{\alpha}}

\newcommand{\on}{\mathtt{ON}}
\newcommand{\off}{\mathtt{OFF}}

\newcommand{\df}{\coloneqq}
\let\defeq=\df
\newcommand\Iffdef{\;\mathrel{\mathord{:}\mathord{\Longleftrightarrow}}\;}

\numberwithin{equation}{section}

\theoremstyle{plain}
 \newtheorem{theorem}{Theorem}[section]
 \newtheorem{lemma}[theorem]{Lemma}

\theoremstyle{definition}
\newtheorem{definition}[theorem]{Definition}
\newtheorem{example}[theorem]{Example}

\definecolor{firstcolumn}{HTML}{2A629A}
\definecolor{labelsraw}{HTML}{2A629A}
\definecolor{header}{HTML}{FEAE6F}
\definecolor{lines}{HTML}{01204E}
\definecolor{mildorange}{HTML}{FF7F3E}
\definecolor{mildgold}{HTML}{FFDA78}

\newcolumntype{s}{>{\columncolor{firstcolumn!25}} c}
\arrayrulecolor{lines}

\title{Towards a logic of affordances}
\author{Rafa\l\ Gruszczy\'nski, Paula Mench\'on, Ivo D\"untsch, and G\"unther Gediga}

\date{}

\address{Rafa\l\ Gruszczy\'nski, \textsc{Orcid:} 0000-0002-3379-0577\\
Department of Logic\\
Nicolaus Copernicus University in Toru\'n\\
Poland}
\email{gruszka@umk.pl}

\urladdr{https://www.umk.pl/~gruszka/}

\address{Paula Mench\'on, \textsc{Orcid:} 0000-0002-9395-107X\\
Department of Logic\\
Nicolaus Copernicus University in Toru\'n\\
Poland\\
and
Universidad Nacional del Centro de la Provincia de Buenos Aires\\
Tandil, Buenos Aires\\
Argentina
}

\email{mpmenchon@nucompa.exa.unicen.edu.ar}

\address{Ivo D\"untsch, \textsc{Orcid:} 0000-0001-8907-2382\\
Department of Computer Science\\
Brock University\\	
St Ca\-tha\-ri\-nes, Ontario\\
Canada}

\email{duentsch@brocku.ca}

\urladdr{https://www.cosc.brocku.ca/~duentsch/}

\address{G\"unther Gediga\\
Institut f\"ur Evaluation und Marktanalysen\\
Bissendorf\\
Germany
}

\email{gediga@eval-institut.de}

\begin{document}

\begin{abstract}
    We aim to construct a formal theory of affordances seen as ternary relations. Beginning with a characterization of affordances proposed by James J. Gibson, and utilizing the tools provided by Zdzisław Pawlak's information systems and rough sets, we construct a mathematically precise definition of both crisp and rough affordances. Then, we analyze modal and approximation operators that enable reasoning about affordances in both scenarios.

    \medskip


    \medskip

    \noindent\textsc{Msc:} 03B52, 03B45
\end{abstract}

\maketitle

\thispagestyle{empty}

\section{Introduction}

The concept of \emph{affordance} was introduced by the American psychologist James J. Gibson in \cite{gibson1979}. In his own words

\begin{quote}
\emph{The affordances of the environment are what it offers the animal, what it provides or furnishes, either for good or ill \ldots I mean by it something that refers to both the environment and the animal in a way that no existing term does. It implies the complementarity of the animal and the environment.}  \cite[p. 119]{gibson1979}
\end{quote}

In a way, affordances are action possibilities that are ``offered'' to an actor by an environment. Some examples of affordances are \cite{stoffregen2003}:

\begin{itemize}
\item If a gap in a wall has a certain size relative to the size of a person, the gap affords passage.
\item If a surface is rigid relative to the weight of an animal, it affords stance and perhaps traversal.
\item If a ball falls with a certain velocity relative to a person’s running speed, then it affords catching.
\item If the time gap between the passage of successive cars on a road is greater than the time needed for a pedestrian to cross, then safe crossing is afforded.
\item If a stair is a certain proportion of a person’s leg length, it affords climbing.
\end{itemize}

It is important to note that in Gibson's world the objects in the environment are not conceptual in the first place (``This is a bucket'') but obtained from the concrete physical features of the visual field (after recognizing invariants etc.). The name ``bucket'' is just a label arising from an affordance.

Affordances as binary relations have been analyzed and studied from various perspectives and within various branches of science in recent years:
\begin{itemize}
\item Human factors and design,
\item Robotics and autonomous systems,
\item Virtual reality and augmented reality,
\item Sports and rehabilitation,
\item Cognitive science, neuroscience,
\item And more \ldots
\end{itemize}

An overview of history and recent developments of affordances is \cite{mhk24}.

In the sequel, we would like to change the perspective on affordances and treat them not as binary but as ternary phenomena. It is a common feature of the surrounding world that we observe interactions between actors and objects in environments:
\begin{itemize}
    \item A dog fetches a stick on a meadow,
    \item A basketball player dunks a~ball on a playground,
    \item A robot serves a meal in a restaurant,
    \item A student solves a test in a classroom,
    \item A mathematician writes a proof on a blackboard in their office,
    \item An autonomous car is delivering a group of people to a hotel in a city.
\end{itemize}
In each of the scenarios, we can distinguish four components: an actor, an object, an environment, and an action being performed. However, not every action can be performed by any actor, or with any object, or in any environment. To dunk a ball, an actor has to be sufficiently tall and agile, have a~proper ball that has suitable size and weight, and be on a playground with a basketball hoop. Thus, properties of the three components of an action must be taken into account.

In this work, we intend to conceptualize affordances as ternary relations between actors, objects, and environments, taking into account their properties.\footnote{Some work toward affordances as ternary phenomena (though not necessarily relations in the standard logical sense) was carried out in~\cite{sahin2007aon}.} The main goal is to construct a framework for building logic (or logics) for reasoning with affordances. Our pre-theoretical definition of affordance is as follows:
\begin{definition}\label{def:afford}
    An \emph{affordance} is an action possibility formed by the relationship among an actor, an object, and their environment.\QED
\end{definition}

The relevant components of our approach to this work are as follows:
\begin{itemize}
    \item Actors and objects have various properties, such as height, weight, color, agility, abilities, disabilities, skills, shape, and so on.
    \item Environments also have properties that may be elementary (in some sense), like flat, bumpy, stifling, hot, wet, rainy, or, more complex, adapted for playing basketball, multi-level, equipped with an elevator, dangerous for health, to mention a~few.
\end{itemize}

\section{Information systems}\label{sec:InfSys}

It is not our goal to find out what affordances <<really>> are, but rather to look at them (whatever these are) from a ternary perspective, which we find natural. For us, affordances are ternary relations that hold among actors with properties, objects with properties, and environments with properties; and the aim is to build an operational concept of an affordance that will allow for applications of algebraical and logical tools. We are not going, and don't want, to judge whether this reveals the ``essence'' of affordances. But we want to put forward the tools that will allow us to speak and reason with affordances in an efficient way.

A simple way of organizing knowledge are property systems introduced, among others, by \citet{vak98}. A \emph{property system} (\emph{P-system}) is a structure $\klam{U, V, f}$, where $U$ is a non-empty set whose elements are called \emph{objects}, $V$ is a set whose elements are called \emph{properties}, and $f\colon U \to 2^V$ is a mapping called an \emph{information function}; we do not require that $f(x) \neq \z$. A statement $a \in f(u)$ can be interpreted as ``Object $u$ possesses property $a$''. If $U$ is finite, then a property system is definitionally equivalent to a formal dyadic context of \citet{wille82}, by observing that the function $f\colon U \to 2^V$ can be replaced by a relation $R_f \subseteq U \times V$, where $u \mathrel{R_f} a$ \tiff $a \in f(u)$ which has the same informational content.\footnote{At this stage of our investigation we suppose that we have a correct description of the world, i.e. what we observe is true.}

If we think of a property system as describing possible states of an attribute---such as ``color'' or ``language spoken''---we extend it by the definition of an aggregate structure: An \emph{attribute system} (\emph{A-system}) \cite{vak98} is a structure $\S \df \klam{U, \Omega, \set{V_p: p \in \Omega}, f}$ where
\begin{enumerate}
\item $U$ is a non-empty set of objects,
\item $\Omega$ is a set of property labels or attributes, and $V_p$ is a set of possible values of $p \in \Omega$,
\item $f\colon U \times \Omega \to \bigcup_{p \in \Omega} 2^{V_p}$ is a choice function, where $f(x,p) \subseteq V_p$. Equivalently, we may define $f\colon U \to \prod_{p \in \Omega} 2^{V_a}$.
\end{enumerate}
So, if $p\in\Omega$ is a property label \emph{weight}, then $V_{\mathit{weight}}$ may be a set of rational numbers in some interval that can serve as numerical expressions of the weight of an object (e.g., in kilograms or pounds), or any value that makes sense. It could also be some aggregated value such as ``low'', ``medium'', ``high'' etc.

We call $\klam{U, \Omega, \set{V_p: p \in \Omega}}$ the \emph{skeleton} of $\S$. The product $U \times \prod_{p \in \Omega} 2^{V_p}$ collects all possible vectors of value sets that can be associated with some element of $U$. An information function now picks one element from $\prod_{p \in \Omega} 2^{V_p}$ for each $x \in U$.

Clearly, an attribute system can be regarded as a collection of the property systems. \citet{dgo_ras} introduce relational attribute system as a relational form of attribute systems, and present formalizations of various interpretations of a statement such as $x \in f(x,p)$. These are closely related to databases with incomplete information investigated by \citet{Lipski1977}. A deduction system for reasoning in relation algebras derived from relational attribute systems was presented in \cite{dgo_ras2}.

An element $x \in U$ is called \emph{deterministic}, if $\card{f(x,p)} \leq 1$ for all $p \in \Omega$ or every projection of the vector attribute to $x$ is either an empty set or a singleton subset of $V_p$. The set of all deterministic elements of \S is denoted by $D_\S$. The characterization stems from the fact that the role of the choice function is to narrow down the possibilities for the values of $p$ with respect to the object $x$. If $|f(x,p)|=1$, then we know the exact value of the property $p$ for $x$, or if $|f(x,p)|=\emptyset$, then we know that $x$ does not have this property at all. This is why we call a system with $|f(x,p)|\leqslant 1$ deterministic. If $\card{f(x,p)} \geq 2$ the set has different possibilities of interpretation, see \cite{dgo_ras}.

If $U$ is finite and $U = D_\S$, then \S is called an \emph{information system} (in the sense of \citet{paw82}). In this case we may simplify the definition of the information function to have codomain $\bigcup_{p \in \Omega} V_p$. \citet{vak98} calls these systems ``ontological information systems'', and provides logical counterparts of P-systems and A-systems based on Scott's deduction systems. He continues to develop modal logics for similarity relations on P-systems and A-systems and shows completeness with respect to finite models, continued in \cite{Vakarelov2005}.

\pagebreak

\begin{example}\label{ex:actors}
    Below
\begin{itemize}
    \item $A$ is the domain, a finite set of people.
    \item $\Omega\defeq\{\mathrm{height},\mathrm{agility}\}$ is the set of property labels.
    \item $V_{\mathrm{height}}\defeq\{150,151,\ldots,210\}$ and $V_{\mathrm{agility}}\defeq\{\mathsf{L,S,H,P}\}$ where the four letters describe the level of agility: low, standard, high, professional.
    \item $f\colon A\to V_{\mathrm{height}}\times V_{\mathrm{agility}}$.
\end{itemize}

\medskip

\begin{center}
\begin{tikzpicture}
    \node [shape=rectangle,inner sep=0pt] (act)  {
\begin{tabular}{ |s|cc| }
\hline
\rowcolor{header!25} \multicolumn{3}{|c|}{\textsc{\large Information system for actors}} \\
\hline
\rowcolor{labelsraw!25}\diagbox{$A$}{$\Omega_A$} & height (cm) & agility \\
\hline
$a_1$ & 201 & L  \\
$a_2$ & 155 & P\\
$a_3$ & 189 & H\\
$a_4$  & 190 & L\\
$a_5$ &  178 & S\\
$a_6$ & 181 & P\\
$a_7$ & 190 & L \\
$a_8$ & 201 & L  \\
$a_9$ &  178 & S\\
\hline
\end{tabular}
};
\end{tikzpicture}
\end{center}

\smallskip

\noindent All in all, we have nine actors that are described to us by two properties only.\QED
\end{example}

\begin{example}\label{ex:playgrounds}
The properties we consider do not have to be in any way simple or elementary. We can consider complex features, such as \emph{being adapted to playing basketball}. Below, we have an information system for playgrounds, each of which is characterized in terms of three complex attributes: being adapted or not (in symbols 1 or 0) to playing basketball, soccer, and volleyball.

\medskip

    \begin{center}
\begin{tabular}{ |s|ccc| }
\hline
\rowcolor{header!25} \multicolumn{4}{|c|}{\textsc{\large Information system for playgrounds}} \\
\hline
\rowcolor{labelsraw!25}\diagbox{$E$}{$\Omega_{E}$} & basketball & soccer & volleyball \\
\hline
$e_1$ & 1 & 1 & 1\\
$e_2$  & 1 & 1 & 0\\
$e_3$ & 1 & 0 & 1\\
$e_4$ & 1 & 0 & 0\\
$e_5$ & 0 & 1 & 1\\
$e_6$ & 0 & 1 & 0\\
$e_7$ & 0 & 0 & 1\\
$e_8$ & 0 & 0 & 0\\
$e_9$ & 1 & 1 & 1\\
$e_{10}$ & 0 & 1 & 1\\
\hline
\end{tabular}
\end{center}\QED
\end{example}

\begin{example}
    In this example, let $e_i$ be a plate with $2^i$ diodes, for $i\in\{1,2,3,4\}$. We take $E\df\{e_1,e_2,e_3,e_4\}$, and treat each $e_i$ as an independent environment. This time we put
    \[
    \Omega^E\df\{\mathrm{state}_1,\mathrm{state}_2,\mathrm{state}_3,\mathrm{state}_4\},
    \]
    where each label is interpreted as the state of one of the four environments. For each $i$, $V^{E}_{\mathrm{state}_{i}}$ is the power set of the set diodes of $e_i$. E.g., Since $e_4$ is equipped with sixteen diodes, $V^{E}_{\mathrm{state}_{4}}$ has $2^{16}=65536$ elements, each one indicating which of the diodes are on, and which are off. Thus, $\emptyset$ is the value indicating that all the diodes are off, while the value $\{d^{e_4}_1,d^{e_4}_2\}$ says that $d^{e_4}_1,d^{e_4}_2$ are both on, and the remaining two are off. Each information function $f_E\colon E\to\prod_{i\leqslant 4}V^{E}_{\mathrm{state}_{i}}$ describes conjunctively the states in which all four environments are simultaneously. For example, the function $f_E$ with the following values
    \[
        f_E(e_1)=\{d^{e_1}_1\}\quad f_E(e_2)=\emptyset\quad f_E(e_3)=\{d^{e_3}_3,d^{e_3}_7\}\quad f_E(e_4)=\{d^{e_4}_2,d^{e_4}_4,\ldots,d^{e_4}_{16}\},
    \]
    corresponds to the configuration of four environments where in $e_1$ one diode is active, in $e_2$ no diode is active, in $e_3$ two diodes are active, and in $e_4$ eight. Even in such a relatively simple environment we have
    \[
        (4\cdot 16\cdot 256\cdot 65536)^4 = 2^{30}
    \]
    functions $f_E$ and as many information systems with the skeleton $\klam{E,\Omega^E,\{V^E_{\mathrm{state}_{i}}\mid i\leqslant 4\}}$. Observe that the skeleton from this example can be obtained from four simpler skeletons of the kind described in Example~\ref{ex:IS-four-diodes}.  The difference lies in the fact that in this example we can look into the structure of the elements of each $e\in E$.\QED
\end{example}

\subsection{Information systems for environments}

There are two stances we can take with respect to understanding of the information system $\I_E$ related to environments and their features:
\begin{enumerate}[label=(E\arabic*),ref=E\arabic*]\label{infdiv}
    \item\label{it:E1} We assume that there is one underlying environment (or context) in some pre-theoretical sense, the elements of $E$ are <<parts>> of the environment, and labels in $\Omega_E$ are interpreted as features of these parts. In this case, $V^E_p$ for $p\in\Omega_E$ is the set of values for parts of the environment.

\item\label{it:E2} Or we treat the elements of $E$ as different environments that can have different properties (features) labeled by the elements of $\Omega_E$, in which case each $V^E_p$ represents values of the features of the whole environment $e\in E$.
\end{enumerate}

In \eqref{it:E1}, the environment is regarded as a snapshot (or picture) of the environment which is fixed, and actors and entities are ``variable'' in the sense that each row of the information system corresponds to one actor and one entity. In \eqref{it:E2}, each row corresponds to one environment. These can be integrated by supposing that each row consists of a description of the environment as in \eqref{it:E1}. In this sense, \eqref{it:E1} is a special case of \eqref{it:E2}. We could have similar snapshots for actors and entities. Observe that the main difference is in the independence of rows of tables in information systems.

\begin{example}\label{ex:IS-four-diodes}
    Suppose our environment is a small plate with four diodes $d_1,d_2,d_3,d_4$. Each diode can be in two states, $\off$ or $\on$ (emitting light). In this case the information system $\I_E$ has the following components:
    \begin{enumerate}[label=(\roman*),itemsep=3pt]
        \item $E\df\{d_1,d_2,d_3,d_4\}$.
        \item $\Omega_E\df\{\mathrm{state}\}$, one label interpreted as the state of a diode.
        \item $V^E_{\mathrm{state}}\df\{\on,\off\}$, i.e., the property \emph{state} has two possible values.
        \item $f^A\colon E\to V^E_{\mathrm{state}}$, i.e., one of the sixteen functions that ``say'' which diodes are on and which are off.
    \end{enumerate}
    Consequently, we have sixteen different information systems based on the skeleton \[\klam{E,\Omega_E,\{V^E_{\mathrm{state}}\}}.\] Here, there is one environment, and the elements of $E$ are parts of this environment. So this example is a representative of the perspective \eqref{it:E1} above.\QED
\end{example}

\begin{example}
    In this example, which is related to \eqref{it:E2}, let $e_i$ be a plate with $2^i$ diodes, for $i\in\{1,2,3,4\}$. We take $E\df\{e_1,e_2,e_3,e_4\}$, and treat each $e_i$ as an independent environment. This time we put
    \[
    \Omega_E\df\{\mathrm{state}_1,\mathrm{state}_2,\mathrm{state}_3,\mathrm{state}_4\},
    \]
    where each label is interpreted as the state of one of the four environments. For each $i$, $V^{E}_{\mathrm{state}_{i}}$ is the power set of the set diodes of $e_i$. E.g., Since $e_4$ is equipped with sixteen diodes, $V^{E}_{\mathrm{state}_{4}}$ has $2^{16}=65536$ elements, each one indicating which of the diodes are on, and which are off. Thus, $\emptyset$ is the value indicating that all the diodes are off, while the value $\{d^{e_4}_1,d^{e_4}_2\}$ says that $d^{e_4}_1,d^{e_4}_2$ are both on, and the remaining two are off. Each information function $f_E\colon E\to\prod_{i\leqslant 4}V^{E}_{\mathrm{state}_{i}}$ describes conjunctively the states in which all four environments are simultaneously. For example, the function $f_E$ with the following values
    \[
        f_E(e_1)=\{d^{e_1}_1\}\quad f_E(e_2)=\emptyset\quad f_E(e_3)=\{d^{e_3}_3,d^{e_3}_7\}\quad f_E(e_4)=\{d^{e_4}_2,d^{e_4}_4,\ldots,d^{e_4}_{16}\},
    \]
    corresponds to the configuration of four environments where in $e_1$ one diode is active, in $e_2$ no diode is active, in $e_3$ two diodes are active, and in $e_4$ eight. Even in such a relatively simple environment we have
    \[
        (4\cdot 16\cdot 256\cdot 65536)^4 = 2^{30}
    \]
    functions $f_E$ and as many information systems with the skeleton $\klam{E,\Omega_E,\{V^E_{\mathrm{state}_{i}}\mid i\leqslant 4\}}$. Observe that the skeleton from this example can be obtained from four simpler skeletons of the kind described in Example~\ref{ex:IS-four-diodes}.  The difference lies in the fact that in this example we can look into the structure of the elements of each $e\in E$.\QED
\end{example}

Here, we are not going to adjudicate on which approach is ``better'' or more suitable for the formal theory of affordances. We only point to the fact that below constructing various examples we adhere to \eqref{it:E2} perspective on environments.

\section{Rough sets}

An information system differs from a relational database in such a way that duplicate rows of attribute values are allowed. The idea behind this assumption is the observation that we can observe the world only up to a certain granularity which depends on the available (or interesting) information. With such limited knowledge it is reasonable to assume that more than one object has the same associated row of attribute values. Furthermore, with only the knowledge given by the information system we cannot distinguish objects with the same value vector. This leads to the following definition of an \emph{indiscernibility relation} $\theta$:
\begin{definition}
Given an information system $\I\defeq\klam{U,\Omega,\{V_{p}\mid p\in\Omega\},f}$, for all $a,b \in U$
\[
a \mathrel{\theta} b \Iffdef f(a)=f(b),
\]
that is, $a$ and $b$ have the same value vector. In this case, we say that $a$ is similar to $b$ (or that $a$ and $b$ are similar). Obviously, in the case of the deterministic systems, $a$ and $b$ are similar if they have precisely the same values for all the property labels in $\Omega$. \QED
\end{definition}
Clearly, $\theta$ is an equivalence relation on $U$ whose classes we denote by $\nec{a}_\theta$ or simply by $\nec{a}$ if no confusion can arise. By means of $\theta$, we can make use of a concept of a \emph{rough set} to incorporate the granularity of the world into our formalism. Given a subset $X$ of $U$ and an element $a \in U$, when can we say that certainly $a \in X$ or certainly $a \not\in X$? Since we know the world only up to the classes of $\theta$ we can say that
\begin{align*}
a \text{ is certainly in } X &{}\Iffdef [a] \subseteq X, \\
a \text{ is certainly not in } X &{}\Iffdef [a] \cap X = \z, \\
a \text{ is possibly in } X &{}\Iffdef [a] \cap X \neq \z \tand [a] \cap - X \neq \z.
\end{align*}

The regions of certainty with respect to $X$ are $\underline{X}$ and $U \setminus \overline{X}$, and the boundary of $X$ is $\overline{X} \setminus \underline{X}$, see Figure \ref{fig:rough}.

\begin{figure}[htb]
  \centering
\includegraphics[width=0.5\textwidth]{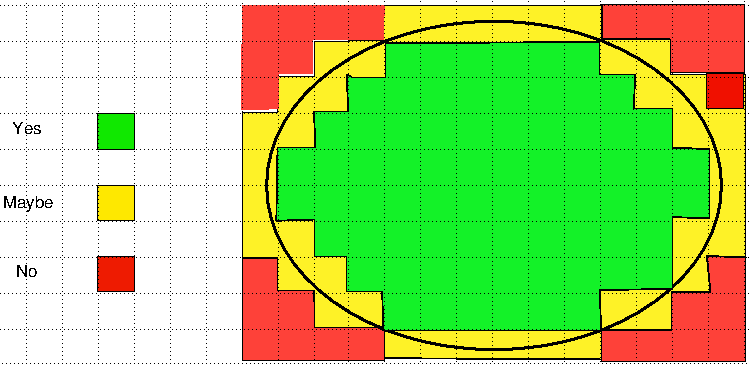}\caption{Rough areas of knowledge}\label{fig:rough}
\end{figure}

\begin{definition}
Formally, a \emph{rough set} is a pair $\klam{\underline{X}, \overline{X}}$\footnote{For a brief overview of the evolvement of rough sets see \cite{mod25}.} such that $X \subseteq U$ and
\begin{xalignat}{2}
\overline{X} &{}\defeq \set{z \in U: \nec{z} \cap X \neq \z}, &&\text{upper approximation,} \label{def:upp} \\
\underline{X} &{}\defeq \set{z \in U: [z] \subseteq X}, &&\text{lower approximation.} \label{def:low}
\end{xalignat}
The \emph{boundary} of $X$ is the set $\oX\setminus\uX$. Upper and lower approximations are closely related to the possibility, respectively, necessity operators of modal logic \cite{dg_approx}.\QED
\end{definition}

The basic assumption is that everything in this world can only be perceived up to a certain granularity, a finite set $\Omega$ of features (properties, attributes), each of which has a number of ``incarnations'', for example, the feature \emph{colour} may have a set of values $\set{\mathit{blue}, \mathit{green}, \mathit{red}}$ or similar sets. Given this limited view of reality, it follows that items may be indiscernible with the knowledge at hand. Taking the information systems described in Section \ref{sec:InfSys} we can think of the situation of interest as described by an information system, where each object $a \in U$ is described by a vector $f(a)$ (row) of attribute values. Let us look at concrete examples.

\begin{example}\label{ex2}
Consider a description of television sets \cite{dg_noninv2}:
\medskip
\begin{center}
\begin{tabular}{|c|c|c|c|c|} \hline
Type & Price & Guarantee & Sound & Screen \\ \hline\hline
1 & high & 24 months & Stereo & 76 \\ \hline
2&low& 6 months & Mono & 66 \\ \hline
3&low& 12 months & Stereo & 36 \\ \hline
4&medium& 12 months & Stereo & 51 \\ \hline
5&medium& 12 months & Stereo & 51 \\ \hline
6&high& 12 months & Stereo & 51 \\ \hline
\end{tabular}
\end{center}

\medskip

Not all objects of $U$ can be distinguished with knowledge expressed in the system; for example, the TV sets $4$ and $5$ have the same properties, i.e. $f(4) = f(5)$. This induces an equivalence relation (\emph{indiscernibility}) $\theta$ on $U$ where $a \mathrel{\theta} b$ \tiff $f(a) = f(b)$. Given a subset $X$ of $U$ and an element $a \in U$, when can we say that certainly $a \in X$ or certainly $a \not\in X$?
In this example, the partition resulting from the system is
\begin{gather*}
\set{1}, \set{2}, \set{3}, \set{4,5}, \set{6}.
\end{gather*}

If $X = \set{2,3,4}$ then
\begin{align*}
& 2 \text{ is certainly in }X, \\
& 6 \text{ is certainly not in }X, \\
& 4 \text{ is possibly in }X, \\
& 5 \text{ is possibly in }X.
\end{align*}\QED
\end{example}

\begin{example}
    Let us consider the information system for actors from Example~\ref{ex:actors}. Let $X\defeq\{a_1,a_3,a_4,a_5,a_9\}$. Then, since $[a_8]\cap X\neq\emptyset$, we can conclude that $a_8$ is possibly at $X$, i.e., $a_8\in\oX$. Similarly for $a_7$. Yet $a_9$ is certainly in $X$, as $[a_9]\subseteq X$, i.e., $a_9\in\uX$. The same with $a_3$. In consequence
    \[
        \oX=\{a_1,a_3,a_4,a_5,a_7,a_8,a_9\}\qtand\uX=\{a_3,a_5,a_9\}.
    \]
    Therefore, the boundary of $X$ is the set $\{a_1,a_4,a_7,a_8\}$.

Next, consider $X'\defeq\ X \setminus \set{a_3,a_5} = \{a_1,a_4,a_9\}$.   Then, $\underline{X'} = \z$ and we cannot say with certainty that $a \in X'$ for any $a \in A$. Dually, if $X'' = A \setminus X' = \set{a_2,a_3,a_5,a_6,a_7,a_8}$, then $\overline{X''} = A$, and any element of $A$ is possibly in $X''$ with the information at hand.
\QED\end{example}

The previous example shows that approximation operators can yield counterintuitive results. If, say, $\theta$ is an equivalence relation on $A$ with two very large classes $X$ and $Y$, and $x \in X, y \in Y$, then $\underline{A \setminus \set{x,y}} = \z$, and $\overline{\set{x,y}} = A$. To handle such situations, in particular, when rough set methods are used for classification or supervised learning, additional statistical methods are required, such as approximation quality, significance testing, or information theoretic entropy. An overview of such methods can be found in \cite{gd_stattech}.

\section{Operationalizing affordances}\label{sec:OpAfford}

A direction on operationalization of affordances was suggested in \cite{dgl09}:

\begin{itemize}
\item ``A formalization of affordance relations needs to provide crisp and fuzzy structures, mechanisms for spatial and temporal change, as well as contextual modeling. ''
\end{itemize}

The basic setup of an affordance relation consists of a set $U$ of an agent's abilities, a set $E$ of features of the environment, and a binary relation $R \subseteq U \times E$.  \citet[p. 189]{chemero2003} writes
\begin{itemize}
\item ``Affordances \ldots are relations between the abilities of organisms and features of the environment. Affordances, that is, have the structure \textsf{Affords--$\phi$ (feature,ability)}''
\end{itemize}

We expand this notion by regarding an affordance in a first step  as a relation $\phi \subseteq A \times O \times E$ where $\phi(a,o,e)$ is interpreted as
\begin{itemize}
\item ``Entity $o$ affords action $\Act_\phi$ to the actor (or perceiver, agent) $a$ in the environment (context)~$e$''.
\end{itemize}
Above, $\Act_{\phi}$ is thought of as an action related to the triple $\phi$ (further below, we give a~more formal explanation).

The initial notion of an affordance is quite coarse, and all three components require further description. Therefore, we extend the concept as follows: suppose that for a set $A$ of actors, a set $O$ of entities or objects, and a set $E$ of environmental factors we have deterministic information systems

\begin{multline*}
\textstyle\I_A = \left\langle A, \Omega_A, \set{V_p^A: p \in \Omega_A}, f_A: A \to \prod_{p \in \Omega_A}V_p^A\right\rangle, \\[+5pt]
\textstyle\I_O = \left\langle O, \Omega_O, \set{V_p^O: p \in \Omega_O}, f_O: O \to \prod_{p \in \Omega_A}V_p^O\right\rangle, \\[+5pt]
\textstyle\I_E = \left\langle E, \Omega_E, \set{V_p^A: p \in \Omega_E}, f_E: E \to \prod_{p \in \Omega_A}V_p^E\right\rangle.
\end{multline*}
Each of these information systems is interpreted as a description, respectively, of actors, entities, or the environment.

\begin{definition}
An \emph{affordance} is a relation
\begin{equation*}
\phi \subseteq \set{\klam{a, f_A(a)}: a \in A} \times \set{\klam{o, f_O(o)}: o \in O} \times \set{\klam{e, f_E(e)}: e \in E}.
\end{equation*}
Thus, an affordance is a ternary relation that holds among actors with properties, objects with properties, and environments (contexts) with properties. \QED
\end{definition}

See Figure \vref{fig:afford} for a pictorial interpretation. A~tuple in $\phi$ is interpreted as ``action $\Act_\phi$ is afforded for actor $a$ described by $f_A(a)$ by entity $o$ described by $f_O(o)$ in the context $e$ described by $f_E(e)$''. To express this and declutter notation we
agree that by \emph{actor} (resp. \emph{object}, \emph{environment}) we always mean $\klam{a,f_A(a)}$ (resp. $\klam{o,f_O(o)}$, $\klam{e,f_E(e)}$). By \emph{an action} we mean an affordance that has been labeled (given a~name), and such it has been recognized as a~possibility.

Let us note that in the definition of an affordance, we use only deterministic systems. Indeterministic systems may be the next step.

\begin{example}
We identify
\begin{enumerate}[label=(\roman*)]
    \item the set of actors $A$ with the set of people,
    \item the set of objects $O$ with the set of various kinds of balls,
    \item the set of environments $E$ with the set of playgrounds.
\end{enumerate}
We define the sets of property labels for each set above in the following way
    \begin{gather*}
        \Omega_A\df\{\mathrm{agility},\mathrm{height}\}\quad\quad\quad
        \Omega_O\df\{\mathrm{diameter},\mathrm{weight}\}\\
        \Omega_E\df\{\mathrm{basketball},\mathrm{handball},\mathrm{soccer}\}
    \end{gather*}
    The values for the elements of  $V_{\mathrm{height}}^A$ and all $V^O$s are the positive rational numbers, with fairly obvious interpretation. Further
    \[
    V_\mathrm{agility}^A\df\{\mathrm{Low},\mathrm{Standard},\mathrm{High},\mathrm{Pro}\}.
    \]
    If the value is $\mathrm{Low}$ the agility is low (e.g., an actor can walk, jump to very low heights, but not run very fast); if $\mathrm{Standard}$ the actor has a <<standard>> level of agility, i.e., can run, jump (but not very high), etc.; if $\mathrm{High}$, the actor has a college varsity member agility; and if $\mathrm{Pro}$ they are on the professional level.

    $V^E\df\{0,1\}$ where the numbers $0$ and $1$ are Booleans, that is they unequivocally indicate if the environment (the playground) is adapted for one of the three activities: 0 means No, 1 means Yes. For example, for $f_E\colon E\to V_{\mathrm{b}}^E\times V_{\mathrm{h}}^E\times V_{\mathrm{s}}^E$ (where the letters `b', `h` and `s' abbreviate, respectively, `basketball', `handball' and `soccer'), if $f_E(e)=\klam{1,0,1}$ it means that environment is adapted for playing basketball and soccer, but not for handball. Here, we treat \emph{adapted for playing basketball} (or any of two other games) as a kind of high-level, or high-order, or aggregate property which might be described by means of simpler ones, e.g., \emph{being flat}, \emph{being horizontally leveled}, \emph{having at least one basketball hoop}, etc. However, we simplify it as the refinement of a description of an environment (an actor, an object) seems to be dependent on the goal, or the background level of the interpreter of the components of an affordance.

 We can now look at the product of the three information systems and represent an affordance by a set of vectors. In Figure~\vref{fig:dunk}, an affordance $\varphi$ consists of all triples of blue vectors.

Our intention is that the triples represent an affordance allowing for an action of \emph{dunking a ball}. All objects in each information system are concrete entities (not equivalence classes). The actors $a_6$ and $a_{10}$ have the same properties but only one is in the affordance, which means that there are some hidden factors that influence the second actor, for example, knee injury.\QED
\end{example}

Observe that the product of information systems constitutes a relational database in Codd's sense \cite{Codd-ARMODFLSDB}.

\afterpage{%
\clearpage
\begin{landscape}
\thispagestyle{empty}
\vspace*{\fill}
\begin{figure}[htb]
\centering
\begin{tikzpicture}
\tikzset{
            affordance/.pic={
                code={
\draw [ultra thick,decorate,
decoration = {calligraphic brace, raise = 3pt, amplitude = 4pt,mirror}] (0,1.1) --  (0,-0.1) node[pos=0.5,left=10pt,black]{\huge $\phi=$};
\fill [mildorange] (0,1) rectangle (1,0);
\fill [mildgold] (1.5,1) rectangle (2.5,0);
\fill [cyan] (3,1) rectangle (4,0);
\draw [ultra thick,decorate, decoration = {calligraphic brace, raise = 3pt, amplitude = 4pt}] (4,1.1) --  (4,-0.1);
}
}
}


\node [shape=rectangle,scale=0.85,inner sep=0pt] (act)  {
\begin{tabular}{ |s|cccccc| }
\hline
\rowcolor{header!25} \multicolumn{7}{|c|}{\textsc{\large Information system for actors}} \\
\hline
\rowcolor{labelsraw!25}\diagbox{$A/_{\sim}$}{$\Omega_A$} & $p_{1}$ & $p_{2}$ & $p_{3}$ & $p_{4}$ & $p_{5}$ & $p_{6}$ \\
\hline
\rowcolor{mildgold} $[a_1]$ & $v_{a_1}^{p_1}$ & $v_{a_1}^{p_2}$ & $v_{a_2}^{p_3}$ & $v_{a_2}^{p_4}$ & $v_{a_2}^{p_5}$ & $v_{a_2}^{p_6}$ \\
\rowcolor{cyan}
$[a_2]$ & $v_{a_2}^{p_1}$ & $v_{a_2}^{p_2}$ & $v_{a_2}^{p_3}$ & $v_{a_2}^{p_4}$ & $v_{a_2}^{p_5}$ & $v_{a_2}^{p_6}$\\
$[a_3]$   & $v_{a_3}^{p_1}$ & $v_{a_3}^{p_2}$ & $v_{a_3}^{p_3}$ & $v_{a_3}^{p_4}$ & $v_{a_3}^{p_5}$ & $v_{a_3}^{p_6}$ \\
\rowcolor{mildorange} $[a_4]$  & $v_{a_4}^{p_1}$ & $v_{a_4}^{p_2}$ & $v_{a_4}^{p_3}$ & $v_{a_4}^{p_4}$ & $v_{a_4}^{p_5}$ & $v_{a_4}^{p_6}$ \\
$[a_5]$ &  $v_{a_5}^{p_1}$ & $v_{a_5}^{p_2}$ & $v_{a_5}^{p_3}$ & $v_{a_5}^{p_4}$ & $v_{a_5}^{p_5}$ & $v_{a_5}^{p_6}$ \\
$[a_6]$ & $v_{a_6}^{p_1}$ & $v_{a_6}^{p_2}$ & $v_{a_6}^{p_3}$ & $v_{a_6}^{p_4}$ & $v_{a_6}^{p_5}$ & $v_{a_6}^{p_6}$    \\
$[a_7]$ & $v_{a_7}^{p_1}$ & $v_{a_7}^{p_2}$ & $v_{a_7}^{p_3}$ & $v_{a_7}^{p_4}$ & $v_{a_7}^{p_5}$ & $v_{a_7}^{p_6}$ \\
\hline
\end{tabular}
};

\node [shape=rectangle,scale=0.85,inner sep=0pt,right = 1cm of act.north east,anchor=north west] (obj)  {
\begin{tabular}{ |s|ccccccc| }
\hline
\rowcolor{header!30} \multicolumn{8}{|c|}{\textsc{\large Information system for entities}} \\
\hline
\rowcolor{labelsraw!25}\diagbox{$O/_{\sim}$}{$\Omega_O$} & $q_{1}$ & $q_{2}$ & $q_{3}$ & $q_{4}$ & $q_{5}$ & $q_{6}$ & $q_{7}$ \\
\hline
$[o_1]$ & $v_{o_1}^{q_1}$ & $v_{o_1}^{q_2}$ & $v_{o_1}^{q_3}$ & $v_{o_1}^{q_4}$ & $v_{o_1}^{q_5}$ & $v_{o_1}^{q_6}$ & $v_{o_1}^{q_7}$\\
\rowcolor{cyan}
$[o_2]$ & $v_{o_2}^{q_1}$ & $v_{o_2}^{q_2}$ & $v_{o_2}^{q_3}$ & $v_{o_2}^{q_4}$ & $v_{o_2}^{q_5}$ & $v_{o_2}^{q_6}$ & $v_{o_2}^{q_7}$\\
$[o_3]$   & $v_{o_3}^{q_1}$ & $v_{o_3}^{q_2}$ & $v_{o_3}^{q_3}$ & $v_{o_3}^{q_4}$ & $v_{o_3}^{q_5}$ & $v_{o_3}^{q_6}$ & $v_{o_3}^{q_7}$ \\
$[o_4]$  & $v_{o_4}^{q_1}$ & $v_{o_4}^{q_2}$ & $v_{o_4}^{q_3}$ & $v_{o_4}^{q_4}$ & $v_{o_4}^{q_5}$ & $v_{o_4}^{q_6}$ & $v_{o_4}^{q_7}$ \\
\rowcolor{mildgold} $[o_5]$ &  $v_{o_5}^{q_1}$ & $v_{o_5}^{q_2}$ & $v_{o_5}^{q_3}$ & $v_{o_5}^{q_4}$ & $v_{o_5}^{q_5}$ & $v_{o_5}^{q_6}$ & $v_{o_5}^{q_7}$ \\
$[o_6]$ & $v_{o_6}^{q_1}$ & $v_{o_6}^{q_2}$ & $v_{o_6}^{q_3}$ & $v_{o_6}^{q_4}$ & $v_{o_6}^{q_5}$ & $v_{o_6}^{q_6}$ & $v_{o_6}^{q_7}$   \\
$[o_7]$ & $v_{o_7}^{q_1}$ & $v_{o_7}^{q_2}$ & $v_{o_7}^{q_3}$ & $v_{o_7}^{q_4}$ & $v_{o_7}^{q_5}$ & $v_{o_7}^{q_6}$ & $v_{o_7}^{q_7}$ \\
\rowcolor{mildorange} $[o_8]$ & $v_{o_8}^{q_1}$ & $v_{o_8}^{q_2}$ & $v_{o_8}^{q_3}$ & $v_{o_8}^{q_4}$ & $v_{o_8}^{q_5}$ & $v_{o_8}^{q_6}$ & $v_{o_8}^{q_7}$ \\
\hline
\end{tabular}
};

\node [shape=rectangle,scale=0.85,inner sep=0pt,right = 1cm of obj.north east,anchor=north west] (env) {
\begin{tabular}{ |s|ccccccc| }
\hline
\rowcolor{header!30} \multicolumn{8}{|c|}{\textsc{Information system for environments}} \\
\hline
\rowcolor{labelsraw!25}\diagbox{$E/_{\sim}$}{$\Omega_E$} & $r_{1}$ & $r_{2}$ & $r_{3}$ & $r_{4}$ & $r_{5}$ & $r_{6}$ & $r_{7}$ \\
\hline
$[e_1]$ & $v_{e_1}^{r_1}$ & $v_{e_1}^{r_2}$ & $v_{e_1}^{r_3}$ & $v_{e_1}^{r_4}$ & $v_{e_1}^{r_5}$ & $v_{e_1}^{r_6}$ & $v_{e_1}^{r_7}$\\
$[e_2]$ & $v_{e_2}^{r_1}$ & $v_{e_2}^{r_2}$ & $v_{e_2}^{r_3}$ & $v_{e_2}^{r_4}$ & $v_{e_2}^{r_5}$ & $v_{e_2}^{r_6}$ & $v_{e_2}^{r_7}$\\
\rowcolor{cyan} $[e_3]$   & $v_{e_3}^{r_1}$ & $v_{e_3}^{r_2}$ & $v_{e_3}^{r_3}$ & $v_{e_3}^{r_4}$ & $v_{e_3}^{r_5}$ & $v_{e_3}^{r_6}$ & $v_{e_3}^{r_7}$ \\
\rowcolor{mildorange} $[e_4]$  & $v_{e_4}^{r_1}$ & $v_{e_4}^{r_2}$ & $v_{e_4}^{r_3}$ & $v_{e_4}^{r_4}$ & $v_{e_4}^{r_5}$ & $v_{e_4}^{r_6}$ & $v_{e_4}^{r_7}$ \\
$[e_5]$ &  $v_{e_5}^{r_1}$ & $v_{e_5}^{r_2}$ & $v_{e_5}^{r_3}$ & $v_{e_5}^{r_4}$ & $v_{e_5}^{r_5}$ & $v_{e_5}^{r_6}$ & $v_{e_5}^{r_7}$ \\
$[e_6]$ & $v_{e_6}^{r_1}$ & $v_{e_6}^{r_2}$ & $v_{e_6}^{r_3}$ & $v_{e_6}^{r_4}$ & $v_{e_6}^{r_5}$ & $v_{e_6}^{r_6}$ & $v_{e_6}^{r_7}$   \\
$[e_7]$ & $v_{e_7}^{r_1}$ & $v_{e_7}^{r_2}$ & $v_{e_7}^{r_3}$ & $v_{e_7}^{r_4}$ & $v_{e_7}^{r_5}$ & $v_{e_7}^{r_6}$ & $v_{e_7}^{r_7}$ \\
$[e_8]$ & $v_{e_8}^{r_1}$ & $v_{e_8}^{r_2}$ & $v_{e_8}^{r_3}$ & $v_{e_8}^{r_4}$ & $v_{e_8}^{r_5}$ & $v_{e_8}^{r_6}$ & $v_{e_8}^{r_7}$ \\
\rowcolor{mildgold} $[e_9]$  &  $v_{e_9}^{r_1}$ & $v_{e_9}^{r_2}$ & $v_{e_9}^{r_3}$ & $v_{e_9}^{r_4}$ & $v_{e_9}^{r_5}$ & $v_{e_9}^{r_6}$ & $v_{e_9}^{r_7}$ \\
\hline
\end{tabular}
};

\node [right = 0.25cm of act.east] {$\times$};
\node [right = 0.25cm of obj.east] {$\times$};

\pic [below = 2cm of obj,xshift=-3cm] {affordance};

\end{tikzpicture}\caption{The triples of vectors of the same color constitute the affordance $\phi$ and its corresponding action $\Act_\phi$. We identify, respectively, actors, objects, and environments that cannot be distinguished by available properties.}\label{fig:afford}
\end{figure}
\vfill
\end{landscape}

\begin{landscape}
\thispagestyle{empty}
\vspace*{\fill}
\begin{figure}[htb]
\centering
\begin{tikzpicture}
\node [shape=rectangle,scale=0.85,inner sep=0pt] (act)  {
\begin{tabular}{ |s|cc| }
\hline
\rowcolor{header!25} \multicolumn{3}{|c|}{\textsc{\large Information system for actors}} \\
\hline
\rowcolor{labelsraw!25}\diagbox{$A$}{$\Omega_A$} & height (cm) & agility \\
\hline
$a_1$ & 201 & L  \\
$a_2$ & 155 & P\\
\rowcolor{cyan} $a_3$ & 192 & H\\
$a_4$  & 190 & L\\
$a_5$ &  178 & S\\
\rowcolor{cyan} $a_6$ & 181 & P\\
$a_7$ & 190 & L \\
$a_8$ & 201 & L  \\
$a_9$ &  178 & S\\
$a_{10}$ & 181 & P\\
\hline
\end{tabular}
};

\node [shape=rectangle,scale=0.85,inner sep=0pt,right = 1cm of act.north east,anchor=north west] (obj)  {
\begin{tabular}{ |s|cc| }
\hline
\rowcolor{header!25} \multicolumn{3}{|c|}{\textsc{\large Information system for objects}} \\
\hline
\rowcolor{labelsraw!25}\diagbox{$O$}{$\Omega_O$} & diameter (in cm) & weight (in kg) \\
\hline
\rowcolor{cyan} $o_1$ & 24.2 & 0.62 \\
$o_2$ & 48.2 & 0.75\\
\rowcolor{cyan} $o_3$ & 6.74 & 0.06\\
\rowcolor{cyan} $o_4$  & 20 & 0.41\\
$o_5$ &  25.1 & 50\\
\rowcolor{cyan} $o_6$ & 0.4 & 0.02\\
\hline
\end{tabular}
};

\node [shape=rectangle,scale=0.85,inner sep=0pt,right = 1cm of obj.north east,anchor=north west] (env) {
\begin{tabular}{ |s|ccc| }
\hline
\rowcolor{header!25} \multicolumn{4}{|c|}{\textsc{\large Information system for environments}} \\
\hline
\rowcolor{labelsraw!25}\diagbox{$E$}{$\Omega_E$} & basketball & handball & soccer \\
\hline
\rowcolor{cyan} $e_1$ & 1 & 1 & 1\\
\rowcolor{cyan} $e_2$  & 1 & 1 & 0\\
\rowcolor{cyan} $e_3$ & 1 & 0 & 1\\
\rowcolor{cyan}$e_4$ & 1 & 0 & 0\\
$e_5$ & 0 & 1 & 1\\
$e_6$ & 0 & 1 & 0\\
$e_7$ & 0 & 0 & 1\\
$e_8$ & 0 & 0 & 0\\
\rowcolor{cyan} $e_9$ & 1 & 1 & 1\\
$e_{10}$ & 0 & 1 & 1\\
\hline
\end{tabular}
};

\smallskip

\node [right = 0.25cm of act.east] {$\times$};
\node [right = 0.25cm of obj.east] {$\times$};
\end{tikzpicture}\caption{The blue triples of vectors constitute an action $\mathtt{dunk}_{\varphi}$ of dunking a~ball.}\label{fig:dunk}\label{fig:afford1}
\end{figure}
\vfill
\end{landscape}
}

\section{Reasoning with affordances: modal-style operators}\label{sec:modop}

Having conceptualized affordances, we then turn to and focus on the problem of reasoning with them as ternary relations. As we will see, the equivalence classes of entities from three different sorts will play a crucial role in incorporating the granularity of perceptions of the world into the formal framework. In this section however, we begin with the definitions of crisp modal-style operators. As such these are independent from the affordance environment above, and can be laid out in a~completely abstract setting. The logic we envisage is the binary version of an extension of classical propositional modal logic with a possibility operator and its dual necessity operator \cite{Blackburn-et-al-ML}. The expressive power of this logic is limited: even though it can  express relational properties such as reflexivity, symmetry, and transitivity, it cannot express, for example, irreflexivity. To extend the expressive power of the classical logic, a ``sufficiency'' counterpart to the classical necessity operator was independently introduced by \citet{Humberstone-IW} and~\citet{Gargov-et-al}. The algebraic properties of such an operator and representation properties of such a~``possibility-sufficiency logic'' were investigated for the unary case by \citet{Duntsch-et-al-MAATL} and extended to the binary case by \citet{Duntsch-et-al-AMLWBO} where further details can be found. Besides the current investigation of affordances, a further example of an application of this logic is the study of formalizing a notion of \emph{betweeness} \cite{dgm23}.  We use the semantics of the logic developed in \cite{Duntsch-et-al-AMLWBO} to investigate properties of affordance relations. The setup we shall use is the following:

\begin{enumerate}[label=(\roman*)]
\item $A,O,E$ are non-empty sets (we forget about actors--objects--environments interpretation for now),
\item $\phi \subseteq A \times O \times E$ is a ternary relation,
\item $X \subseteq A, Y \subseteq O, Z \subseteq E$.
\end{enumerate}

We define six operators \label{page:possibility-operators}
\begin{align}
\poss{\phi}_{{E}}\colon 2^{A} \times 2^{O} \to 2^{E}, && \poss{\phi}_{{E}}(X,Y) &\df \set{e \in {E}:( X \times Y \times \set{e}) \cap {\phi} \neq \z}, \label{possE}\\
\suff{\phi}_{E}\colon 2^{A} \times 2^{O} \to 2^{E}, && \suff{\phi}_{E}(X,Y) &\df \set{e \in {E}: X \times Y \times \set{e} \subseteq \phi}, \label{suffE}\\
\poss{\phi}_{{O}}\colon 2^{A} \times 2^{E} \to 2^{O}, && \poss{\phi}_{{O}}(X,Z) &\df \set{o \in {O}:( X \times \set{o} \times Y) \cap {\phi} \neq  \z}, \label{possO}\\
\suff{\phi}_{O}\colon 2^{A} \times 2^{E} \to 2^{O}, && \suff{\phi}_{O}(X,Z) &\df \set{o \in {O}: X \times \set{o} \times Y \subseteq \phi}, \label{suffO} \\
\poss{\phi}_{A}\colon 2^{O} \times 2^{E} \to 2^{A}, && \poss{\phi}_{{A}}(Y,Z) &\df \set{a \in {A}:(\set{a} \times Y \times Z) \cap {\phi} \neq \z}, \label{possA}\\
\suff{\phi}_{A}\colon 2^{O} \times 2^{E} \to 2^{A}, && \suff{\phi}_{A}(Y,Z) &\df \set{a \in {A}: \set{a} \times Y \times Z \subseteq \phi}. \label{suffA}
\end{align}

We interpret these as follows:
\begin{itemize}
\item $e \in \poss{\phi}_{E}(X,Y)$ \tiff there are $a \in X, o \in Y$ such that $a$, $o$ and $e$ are $\phi$-related.
\item $e \in \suff{\phi}_{E}(X,Y)$ \tiff for all $a \in X, o \in Y$, the triple $\klam{a,o,e}$ is in $\varphi$. In other words, for objects $a$ and $o$ to be $\varphi$-related to $e$ is sufficient that the pair $\klam{a,o}$ is in $X\times Y$. For this reason, $\suff{\phi}_{E}$ is sometimes called a \emph{sufficiency operator}.
\end{itemize}
The operators for $O$ and $A$ are interpreted in analogous way. From now on, to simplify things, we will focus on the operators for~$E$. All results concerning this transfer directly to the possibility and sufficiency operators for $A$ and $O$ (and other operators defined from these).

To gauge the meaning of the operators in the affordance setting let us consider several examples. Generally, whenever we have an operationalisation of a situation represented by an information table, we assume that the chosen semantics make sense and are trustworthy.

\begin{example}\label{ex:poss-suff}
    Consider the following scenario:
        \begin{itemize}
        \item $A\defeq\text{a set of pets}$, $D\subseteq A$ is the subset of dogs,
        \item $O\defeq\text{a set of pet toys}$, $S\subseteq O$ is the subset of sticks,
        \item $\varphi\defeq\text{the affordance of fetching}$.
        \end{itemize}
        Then,
        \begin{itemize}
         \item $e\in\poss{\varphi}_E(D,S)$ means that it is possible for a dog to fetch a stick in the environment $e$.
         \item $e\in\suffe{\varphi}(D,S)$ means that the environment $e$ affords any dog to carry any stick.
         \end{itemize}\QED
\end{example}
As we can see from the example above, possibility and sufficiency make perfect sense.

The dual operators of $\poss{\phi}$ and $\suff{\phi}$ are, respectively, $\nec{\phi}_{E}$ and $\dusuff{\phi}_{E}$ , where
\begin{align*}
\nec{\phi}_{E}(X,Y) &\df {E} \setminus \poss{\phi}_{E}(A \setminus X, O\setminus Y),\\
\dusuff{\phi}_{E}(X,Y)&\df E\setminus \suff{\phi}_{E}(A \setminus X, O\setminus Y).
\end{align*}
It is straightforward to see that
\begin{align}\label{necessity}
e \in \nec{\phi}_{E}(X,Y) &{}\Iff (\forall u\in A)(\forall v\in O)(\klam{u,v,e} \in \phi \Implies u \in X \tor v \in Y), \\
e \in \dusuff{\phi}_{E}(X,Y) &{}\Iff\left[(A\setminus X)\times(O\setminus{Y})\times\{e\}\right]\cap-\varphi\neq\emptyset.
\end{align}
The first condition expresses that for $\klam{u,v,e} \in \phi$ it is necessary that $u \in X \tor v \in Y$. The dual of the sufficiency operator is, as always, harder to interpret in an intuitive way. In this case, we may say (roughly) that $e \in \dusuff{\phi}_{E}(X,Y)$ holds when there are $a$ beyond $X$ and $o$ beyond $Y$ for which $\klam{a,o,e}\notin\varphi$.

\begin{example}
Let us use one more time the interpretation from Example~\ref{ex:poss-suff}. Then,
\begin{itemize}
    \item $e\in\nece{\varphi}(D,S)$ means that for any animal $a$ that fetches a toy $t$ in the environment $e$, either $a$ is a~dog or $t$ is a~stick. In other words, if you fetch a toy in $e$, you must be either a dog, or you must carry a stick. Putting it more sensibly, in $e$ only dogs can fetch sticks only. It is not easy to imagine that kind of environment, yet the interpretation of the operator is sensible too.
    \item $e \in \dusuff{\phi}_{E}(D,S)$ means that there is a pet $a$ which is not a dog, and there is a toy $o$ which is not a stick, and $a$ does not fetch $o$.
\end{itemize}
The second point confirms the awkwardness of the dual of the sufficiency operator.\QED
\end{example}

Analogously to \cite[Lemma 5.2]{dgm23} we have the conditions for the operators associated with respective sets.

\begin{lemma}
\begin{enumerate}\itemsep+4pt
\item $\suff{\phi}_{E}(X,Y) \subseteq \poss{\phi}_{E}(X,Y)$ for all non-empty $X \subseteq A, Y \subseteq O$.
\item $\suff{\phi}_{O}(X,Y) \subseteq \poss{\phi}_{O}(X,Y)$ for all non-empty $X \subseteq A, Y \subseteq E$.
\item $\suff{\phi}_{A}(X,Y) \subseteq \poss{\phi}_{A}(X,Y)$ for all non-empty $X \subseteq O, Y \subseteq E$.
\end{enumerate}
\end{lemma}
These express that on non-empty sets, a universal statement implies an existential statement. We define additional operators derived from the operators above:
\begin{align}
\suff{\phi}_{E}^*(X,Y) &\df \suff{\phi}_{E}(A \setminus X, O \setminus Y), \\
\necu{\phi}_{E}(X,Y) &\df \nec{\phi}_{E}(X,Y) \cap \suff{\phi}_{E}^*(X,Y).
\end{align}
The dual of $\necu{\phi}_{E}$ is denoted by $\possu{\phi}_{E}$, so that
\begin{gather}
\possu{\phi}_{E}(X,Y) = \poss{\phi}_{E}(X,Y) \cup (E \setminus \suff{\phi}_{E}(X,Y)).
\end{gather}
and therefore,
\begin{equation}\label{possu}
\begin{split}
e \in \possu{\phi}_{E}(X,Y) &{}\Iff (X \times Y \times \set{e}) \cap \phi \neq \z \tor (X \times Y \times \set{e}) \not\subseteq \phi\\
&{}\Iff (X \times Y \times \set{e}) \cap \phi \neq \z \tor (X \times Y \times \set{e}) \cap-\phi\neq\z.
\end{split}
\end{equation}
Analogously, we define the corresponding operators for $O$ and $A$.

\begin{example}
    Let us look at the concrete example of a~ternary relation in the form of betweenness that was studied by us in the setting of possibility-sufficiency algebras in \cite{dgm23}. $B(a,o,e)$ is interpreted as follows: the point $o$ is between (in the non-strict sense) the points $a$ and $e$. Then $e\in\possu{B}(X,Y)$ \tiff either there are $x\in X$ and $y\in Y$ such that $B(x,e,y)$ or either there are $x\in X$ and $y\in Y$ such that $\neg B(x,e,y)$. Then $\possu{B}(X,Y)=\emptyset$ \tiff at least one of the sets is empty, and $\possu{B}(X,Y)$ is the set of all points otherwise.\QED
\end{example}

\begin{example}
    Again, let us check the semantics of $\possu{\phi}_{E}$ from the perspective of the interpretation from Example~\ref{ex:poss-suff}. The statement $e\in\possu{\phi}_{E}(D,S)$ means that either there are a dog $a$ and a stick $o$ such that $a$ fetches $o$, or there are a dog $a$ and a stick $o$ such that $a$ doesn't fetch $o$. Thus, in the latter case, at least there are a dog and a stick (and maybe just the dog is gnawing at the stick).\QED
\end{example}

In general, the statement $e \in \possu{\phi}_{E}(X,Y)$ says: either there are $a\in X$ and $o\in Y$ such that $\klam{a,o,e}\in\varphi$ or there are $a\in X$ and $o\in Y$ such that $\klam{a,o,e}\notin\varphi$. In consequence, for any $X\neq\emptyset\neq Y$, $\possu{\phi}_{E}(X,Y)=E$. Clearly, if $e\in E$, then if $(X \times Y \times \set{e}) \cap \phi=\emptyset$, it must be the case that $(X \times Y \times \set{e}) \subseteq-\phi$. So, by the assumption that both $X$ and $Y$ are non-empty we obtain that $(X \times Y \times \set{e}) \cap-\phi\neq\z$. Thus we have
\begin{equation}
\possu{\phi}_{E}(X,Y)=\begin{cases}
    E\quad\text{if $X\neq\emptyset\neq Y$},\\
    \emptyset \quad\text{otherwise}.
\end{cases}
\end{equation}
So, by duality, we have
\[
\necu{\phi}_{E}(X,Y)=\begin{cases}
    \emptyset\quad\text{if $X\neq A$ and $Y\neq O$},\\
    E \quad\text{otherwise}.
\end{cases}
\]
From this we obtain lemmas~\ref{lem:all}, \ref{lem:about-[phi]]_E} and \ref{lem:about-<phi>>_E} below.

\begin{lemma}\label{lem:all}
$\necu{\phi}_{E}(X,Y) \neq \z$ implies $X = A$ or $Y = O$.
\end{lemma}
\begin{lemma}\label{lem:about-[phi]]_E}
\begin{enumerate}
\item If $X = A$ or $Y = O$, then $\necu{\phi}_{E}(X,Y) = E$.
\item Furthermore,
\begin{align*}
\necu{\phi}_{E}(\z,Y) &=
\begin{cases}
E, &\text{if } Y = O, \\
\z, &\text{otherwise.}
\end{cases} \\
\necu{\phi}_{E}(X,\z) &=
\begin{cases}
E, &\text{if } X = A, \\
\z, &\text{otherwise.}
\end{cases}
\end{align*}
\end{enumerate}
\end{lemma}
\begin{lemma}\label{lem:about-<phi>>_E}
Let $X \subseteq A$ and $Y \subseteq O$. Then,
\begin{align*}
\possu{\phi}_{E}(X,O) &=
\begin{cases}
\z, &\text{if } X = \z, \\
E, &\text{otherwise,}
\end{cases} \\
\possu{\phi}_{E}(A,Y) &=
\begin{cases}
\z, &\text{if } Y = \z, \\
E, &\text{otherwise.}
\end{cases}
\end{align*}
\end{lemma}
The latest of the three lemmas shows that fixing one component to the whole set makes $\possu{\phi}_{E}$ behave like the unary discriminator.

\begin{lemma}\label{necandsuff}
    For every $X\subseteq A$ and $Y\subseteq O$,
    \begin{enumerate}
        \item $\nec{\phi}_{E}(X,O)=\nec{\phi}_{E}(A,Y)=E$ and
        \item $\suff{\phi}_{E}(X,\z)=\suff{\phi}_{E}(\z,Y)= E$.
    \end{enumerate}
\end{lemma}

\begin{proof}
    1. Let $e\in E$. It is obvious that for any $x\in X$ and $o\in O$, if $\klam{x,o,e}\in\varphi$, then $x\in X$ or $o\in O$. From \eqref{necessity}, $e\in \nec{\phi}_{E}(X,O)$, i.e. $\nec{\phi}_{E}(X,O)=E$. In an analogous way we obtain that $\nec{\phi}_{E}(A,Y)=E$.

    2. Let $e\in E$. Since $X\times \emptyset\times \set{e}\subseteq \phi$, by the semantics for the sufficiency operator, $e\in \suff{\phi}_{E}(X,\z)$, i.e. $\suff{\phi}_{E}(X,\z)=E$. It follows analogously that $\suff{\phi}_{E}(\z,Y)= E$.
\end{proof}

The other operators are defined analogously. We note that there are two such discriminators defined on each of $A,O,E$.

The triadic concept operators can be integrated into this setup, see \cite{dg_approx} for the binary case, and we can look at compositions of a binary modality with an operator of triadic contexts which maps a subset to a binary relation. Another way to obtain binary relations is to use cylindrifications, effectively ``forgetting'' one component.

\section{The approximation operators}

Given an information system of actors (or objects, or environments), we can factor it with respect to its set of properties: we identify these vectors with actors and their properties that cannot be distinguished with respect to them. Given $a\in A$, let $[a]$ be the equivalence class of actors. Given $X\subseteq A$, let $\overline{X}$ be the upper approximation of $X$, and $\underline{X}$ be its lower approximation as defined in \eqref{def:upp} and \eqref{def:low}.

Given an affordance $\varphi$, we will abuse the notation writing $([a]\times[b]\times[c])\cap\varphi\neq\emptyset$ to say that there is $\langle x,y,z\rangle\in [a]\times[b]\times[c]$ such that
\[
\klam{\klam{x,f_A(x)},\klam{y,f_O(y)},\klam{z,f_E(z)}}\in\varphi.
\]
The analogous meaning is ascribed to the following statement:
\[
[a]\times[b]\times[c]\subseteq\varphi.
\]
With these, we can put a rough structure on affordances in a~natural way. Firstly,
\[
\klam{\klam{a,f_A(a)},\klam{o,f_O(o)},\klam{e,f_E(e)}}\in\overline{\varphi}\Iffdef([a]\times[o]\times[e])\cap\varphi\neq\emptyset.
\]
Secondly,
\[
\klam{\klam{a,f_A(a)},\klam{o,f_O(o)},\klam{e,f_E(e)}}\in\underline{\varphi}\Iffdef[a]\times[o]\times[e]\subseteq\varphi.
\]
In this way, we define the upper and the lower approximation of the affordance $\varphi$. Before we speak about the meaning of these, let us agree that below we are going to talk about actors instead of vectors of actors and properties, objects instead of vectors of objects and their properties, and environments instead of vectors of environments and their properties. Thus, instead of writing
\[
\klam{\klam{a,f_A(a)},\klam{o,f_O(o)},\klam{e,f_E(e)}}\in\ovarphi
\]
we just write
\[
\klam{a,o,e}\in\ovarphi
\]
and similarly for $\varphi$ and $\uvarphi$. This convention will not lead to any ambiguity, as for now we are not going to make use of properties, which at this moment only come into play when we factorize the sets $A$, $O$, and $E$, respectively.

\begin{definition}
    A \emph{rough affordance} is a pair $\klam{\uvarphi,\ovarphi}$ where $\varphi$ is an affordance relation.\QED
\end{definition}

\begin{lemma}\label{lem:rough-aff-inclusions}
    $\uvarphi\subseteq\varphi\subseteq\ovarphi$.
\end{lemma}
\begin{proof}
    If $\klam{a,o,e}\in\uvarphi$, then $[a]\times[o]\times[e]\subseteq\varphi$, so in particular the triple $\klam{a,o,e}$ is in $\varphi$. Further, if $\klam{a,o,e}$ is in $\varphi$, then it must be the case that $[a]\times[o]\times[e]\cap\varphi\neq\emptyset$. So $\klam{a,o,e}\in\ovarphi$.
\end{proof}

With an affordance $\varphi$ we can associate its upper approximation $\oalpha_{S}(\varphi)$ on every coordinate $S\in\{A,O,E\}$ i.e., we have approximations for, respectively, actors, objects, and end environments:
\begin{align*}
a\in\oalpha_{A}(\varphi)(Y,Z)&{}\Iffdef\left([a]\times\overline{Y}\times\overline{Z}\right)\cap\varphi\neq\emptyset,\\
o\in\oalpha_{O}(\varphi)(X,Z)&{}\Iffdef\left(\overline{X}\times[o]\times\overline{Z}\right)\cap\varphi\neq\emptyset,\\
e\in\oalpha_E(\varphi)(X,Y)&{}\Iffdef\left(\overline{X}\times\overline{Y}\times[e]\right)\cap\varphi\neq\emptyset.
\end{align*}
Let us comment on the relation of the $\oalpha$-operators to possibility operators from page~\pageref{page:possibility-operators}. For this, we fix $S$ to be $E$. Both the possibility operator $\poss{\varphi}_E$ and the approximation operator $\oalpha_{E}(\varphi)$ have the type $2^{A} \times 2^{O} \to 2^{E}$, but they behave differently. The former is crisp and sends the pairs of collections of agents and objects, respectively, to the collection of environments in which agents and objects may $\phi$-interact. The latter is rough, as given a pair $\klam{X,Y}$, where $X$ is a collection of agents and $Y$ is a collection of objects, $\oalpha_E(X,Y)$ is the collection of all those environments  (contexts) $e$ such that an $X$-like agent and $Y$-like object $\phi$-interact in an environment resembling $e$.

In an analogous way, we define $\oalpha$ operators for $\ovarphi$, but, as expected, there is no difference between the upper approximations of $\varphi$ and $\ovarphi$.
\begin{theorem}
For each $S\in\{A,O,E\}$ we have that
\begin{gather}\label{phiover}
\oalpha_S(\varphi)=\oalpha_S\left(\ovarphi\right).
\end{gather}
\end{theorem}
\begin{proof}
    We will show the equality for $E$ only, as the remaining cases are analogous.

    \smallskip

    ($\subseteq$) Let $e\in \oalpha_{E}(\varphi)(X,Y)$, i.e., $\left(\oX\times\oY\times[e]\right)\cap\varphi\neq\emptyset$. Then, by Lemma~\ref{lem:rough-aff-inclusions} it must be the case that $\left(\oX\times\oY\times[e]\right)\cap\ovarphi\neq\emptyset$. So $e\in\oalpha_{E}(\ovarphi)(X,Y)$.

    \smallskip

    ($\supseteq$) For the other direction, assume that $e\in\oalpha_E(\ovarphi)(X,Y)$, that is $\left(\oX\times\oY\times[e]\right)\cap\ovarphi\neq\emptyset$. Let $x\in\oX$, $y\in\oY$, and $z\in[e]$ be such that $([x]\times[y]\times[z])\cap\varphi\neq\emptyset$. Let $\klam{a,o,u}$ be a triple which is in the intersection. But then $[a]=[x]$ and $[o]=[y]$ and $[u]=[e]$, and so $a\in\oX$, $o\in\oY$, and in consequence $\left(\oX\times\oY\times[e]\right)\cap\ovarphi\neq\emptyset$. Thus $e\in \oalpha_E(\varphi)(X,Y)$, as desired.
\end{proof}

To show that the upper approximation operators can be meaningfully applied in the formal theory of affordances let us consider some examples.

\begin{example}\label{ex:upper-approx-with-Robin}
We remain in the realm of pets, pet toys, and playgrounds, with cosmetic changes:
    \begin{itemize}
    \item $A\defeq\text{a set of pets}$,
    \item $O\defeq\text{a set of pet toys}$, $S\subseteq O$ is the subset of sticks,
    \item $E\defeq\text{a set of playgrounds}$, $P\subseteq E$ the set of dog parks,
    \item $\varphi\defeq\text{the affordance of fetching}$,
    \item $r\defeq\text{Robin (the dog of one of the authors)}$.
\end{itemize}
Then, $r\in\overline{\alpha}_A(\varphi)(S,P)$ means that a dog resembling Robin fetches a stick-like object in a surrounding resembling a dog park (dog playground). This perfectly models a situation in which we recognize objects and situations in the world up to a certain level of granularity. We see a dog that resembles the dog we know well, the dog fetches an object that resembles a stick, and the context resembles that of a dog park. A naked-eye observation from a distance or foggy weather may result in precisely such a roughness of perception.\QED
\end{example}

There is no reason, of course, that the machinery should incorporate roughness in all three coordinates. For example, the following operators
\begin{align*}
a\in\oalpha^{1}_{A}(\varphi)(Y,Z)&{}\Iffdef ([a]\times Y\times Z)\cap\varphi\neq\emptyset,\\
a\in\oalpha^{2}_{A}(\varphi)(Y,Z)&{}\Iffdef (\{a\}\times \oY\times Z)\cap\varphi\neq\emptyset,\\
a\in\oalpha^{3}_{A}(\varphi)(Y,Z)&{}\Iffdef (\{a\}\times Y\times \oZ)\cap\varphi\neq\emptyset,\\
a\in\oalpha^{12}_{A}(\varphi)(Y,Z)&{}\Iffdef ([a]\times \oY\times Z)\cap\varphi\neq\emptyset,\\
a\in\oalpha^{13}_{A}(\varphi)(Y,Z)&{}\Iffdef ([a]\times Y\times \oZ)\cap\varphi\neq\emptyset,\\
a\in\oalpha^{23}_{A}(\varphi)(Y,Z)&{}\Iffdef (\{a\}\times \oY\times \oZ)\cap\varphi\neq\emptyset,
\end{align*}
express uncertainty (approximation, roughness) on selected projections of the product $A\times O\times E$.

\begin{example}
    With the interpretation from Example~\ref{ex:upper-approx-with-Robin} the following meanings are associated to the operators:
    \begin{itemize}\itemsep+3pt
        \item $r\in\oalpha^{1}_{A}(\varphi)(S,P)$ means that a dog similar to Robin fetches a~stick in a dog park, that is we are certain what the object and the environment involved in the action are, but we are unsure with respect to the actor.
        \item $r\in\oalpha^{2}_{A}(\varphi)(S,P)$ means that Robin fetches a~stick-like object in a dog park, so the actor and the environment are crisp, but the object is only approximated.
        \item $r\in\oalpha^{3}_{A}(\varphi)(S,P)$ means that Robin fetches a~stick in an environment resembling a~dog park. Here the uncertainty concerns the environment, while we are unequivocal about the actor and the object.
        \item $r\in\oalpha^{12}_{A}(\varphi)(S,P)$ means that in a dog park a~dog similar to Robin fetches an~object similar to a~stick. Thus, the level of roughness rises as we only approximate two out of the three components of the action: the environment is crisp, but the actor and the object are rough.
        \item $r\in\oalpha^{13}_{A}(\varphi)(S,P)$ means that in an environment resembling a~dog park a~dog similar to Robin fetches a~stick. The situation is as above, with certainty shifted to the object, with the two remaining components rough.
        \item $r\in\oalpha^{23}_{A}(\varphi)(S,P)$ means that Robin fetches something similar to a~stick in an environment resembling a dog park. We are sure of the actor's identity, but we can only surmise what the object and environment are.
    \end{itemize}
Playing around with the coordinates, we can express the roughness of various aspects of the scenarios we put into focus. Clearly, approximating all three coordinates results in $\oalpha_A(\varphi)$ operator.\QED
\end{example}

\begin{example} Imagine yourself going through the family album and encountering a~photograph of your maternal grandparents, with your grandfather ($a$) kissing your grandmother ($b$) on the cheek. The background is a medieval town, which you know is in Italy, and which you find similar to Urbino. In this case, assuming that $\varphi$ is the affordance of kissing on the cheek, $\oalpha_E^3(\varphi)(\{a\},\{b\})$ is the collection of all medieval Italian towns that are similar to Urbino.\QED
\end{example}

\begin{example} The manakins are small birds native to Central America distinguished by a~dancing ritual involving a kind of moonwalk, usually performed on a tree branch~\cite{manakin}. An ornithologist fascinated by the phenomenon may naturally ask: \emph{are there similar birds mating in this way?} To find out, they should localize places (environments) where a male shows off in an analogous way. To capture it in our framework, we set the set $A$ of actors, and $O$ of objects to consist of birds (so $A=O$). Let $M$ be the set of manakin males, $F$ of manakin females. Let $\varphi$ be the affordance of moonwalking in front of your potential partner, limited to birds. With this interpretation we have that
\[
\oalpha^{12}_{E}(M,F)=\{e\in E\mid (\overline{M}\times\overline{F}\times\{e\})\cap\varphi\neq\emptyset\}
\]
is the collection of all environments in which birds similar to manakins participate in the dancing ritual depicted in the linked footage, i.e., the places of interest of the curious ornithologist. \QED
\end{example}

For any $i\in\{1,2,3,12,13,23\}$ by Lemma~\ref{lem:rough-aff-inclusions} we have
\[
\oalpha^{i}_{A}(\varphi)(X,Y)\subseteq\oalpha^{i}_{A}(\ovarphi)(X,Y).
\]
However, unlike the unrestricted versions in \eqref{phiover}, the reverse inclusions fail in general, as we show below.
\begin{example}
We stick to the interpretation from Example~\ref{ex:upper-approx-with-Robin}, and we consider the case in which $i\defeq 23$. Suppose that $r\in\oalpha^{23}_{A}(\ovarphi)(S,P)$, i.e.
\begin{xalignat*}{2}
r\in\oalpha^{23}_{A}(\ovarphi)(S,P) &\Iff (\{r\}\times\overline{S}\times\overline{P})\cap\ovarphi\neq\emptyset, &&\text{by Def. of $\oalpha^{23}_{A}$,} \\
&\Iff (\exists s \in S, p \in P)[(\set{r} \times \nec{s} \times \nec{p}) \cap \ovarphi \neq \z], &&\text{by Def. of $\overline{X}$,} \\
&\Iff (\exists s \in S, p \in P)[(\nec{r} \times \nec{s} \times \nec{p}) \cap \varphi \neq \z], &&\text{by Def. of $\ovarphi$.}
\end{xalignat*}
This tells us that some dog resembling Robin ($r' \in \nec{r}$) fetches a stick-like object ($o' \in \nec{o}$) in a dog-park-like environment ($p' \in \nec{p}$). If $\card{\nec{r}} \geqslant 2$, then we cannot say for sure that the dog $r'$ that resembles Robin is, indeed, Robin. To see this, observe that by the definition we have
\[
r\in\oalpha^{23}_{A}(\varphi)(S,P) \Iff (\{r\}\times\overline{S}\times\overline{P})\cap\varphi\neq\emptyset.
\]
Observe that in the equivalence above $\varphi$ is used and not its upper approximation, and that the first projection is fixed. Therefore, we may only take into account $\{r\}$, but in general not its equivalence class $[r]$. In the fleshy entourage of the example, we attribute certainty to the actor's identity, i.e., Robin. Thus, passing from $r\in\oalpha^{23}_{A}(\ovarphi)(S,P)$ to $r\in\oalpha^{23}_{A}(\varphi)(S,P)$ expressed in natural language is like saying ``From the fact that a dog similar to Robin is involved in an action, I can conclude that Robin himself is''. But, as we observed, it cannot be done if the collection of dogs similar to Robin has more than one element.
\QED\end{example}

A generalization of the reasoning from the example above shows that for any $i\in\{1,2,3,12,13,23\}$
\[
\oalpha^{i}_{A}(\ovarphi)(X,Y)\subseteq \oalpha_{A}(\varphi)(X,Y).
\]
Thus, as we see, the operator $\oalpha_{S}(\varphi)$ for any $S\in\{A,O,E\}$ gives the largest approximation.

In a complete analogy with the previous section, we may define the lower approximation operators (one for each component) in the following way:
\begin{align*}
a\in\ualpha_{A}(\varphi)(Y,Z)&{}\Iffdef[a]\times\uY\times\uZ\subseteq\varphi,\\
o\in\ualpha_{O}(\varphi)(X,Z)&{}\Iffdef\uX\times[o]\times\uZ\subseteq\varphi,\\
e\in\ualpha_E(\varphi)(X,Y)&{}\Iffdef\uX\times\uY\times[e]\subseteq\varphi.
\end{align*}
\begin{theorem}
For each $S\in\{A,O,E\}$ we have
\[
\ualpha_S(\varphi)(X,Y)=\ualpha_S(\uvarphi)_S(X,Y).
\]
\end{theorem}
\begin{proof}
    This time, the inclusion $\supseteq$ is immediate in light of Lemma~\ref{lem:rough-aff-inclusions}. For $\subseteq$, assume that $e\in\ualpha_E(\varphi)(X,Y)$, i.e.,  $\uX\times\uY\times[e]\subseteq\varphi$. Let $\klam{a,o,f}$ be from $\uX\times\uY\times[e]$. Thus $[a]\subseteq X$, $[o]\subseteq Y$, and $[f]=[e]$. So, $[a]\subseteq\uX$, $[o]\subseteq\uY$, and $[a]\times[o]\times[e]\subseteq\varphi$. Therefore, $\klam{a,o,e}\in\uvarphi$. As a result, $\uX\times\uY\times[e]\subseteq\uvarphi$, and $e\in \ualpha_E(\uvarphi)(X,Y)$.
\end{proof}

However, as we are mainly interested in upper approximations epitomizing roughness and imprecision phenomena, here, we do not delve into properties of the lower approximation operators leaving this topic for further work.

\section{Mixing modal and approximation operators}

From the sections above, we draw the following conclusion. The operators that make most sense in the setting of affordances are
\begin{itemize}
        \item possibility,
        \item sufficiency,
        \item upper $\varphi$-approximation.
    \end{itemize}
As we defined the possibility and sufficiency operators associated with an affordance $\varphi$, we can analogously introduce such operators for the lower and upper approximations of $\varphi$. Below, we point to dependencies between different operators, including lower and upper approximations for~$\varphi$. As before, we distinguish the set $E$, which does not harm the generality of the exposition.
\begin{align}
    \poss{\overline{\varphi}}_E(X,Y)&{}=\oalpha_E(\varphi)(X,Y),\label{eq:poss-up-approx-1}\\
    \overline{\poss{\varphi}_E(\overline{X},\overline{Y})}&{}=\oalpha_E(\varphi)(X,Y),\label{eq:poss-up-approx-2}\\
    \suff{\uvarphi}_{E}(X,Y)&{}\subseteq\ualpha_{E}(\varphi)(X,Y),\label{eq:suff-lo-approx-1}\\
    \underline{\suff{\varphi}_E(\underline{X},\underline{Y})}&{}=\ualpha_{E}(\varphi)(X,Y).\label{eq:suff-lo-approx-2}
\end{align}
\begin{proof}
\eqref{eq:poss-up-approx-1} Let $e\in \poss{\overline{\varphi}}_E(X,Y)$. By definition, $(X\times Y\times \set{e})\cap \overline{\varphi}\neq \emptyset$. It follows that there are $x\in X$ and $y\in Y$ such that $\klam{x,y,e}\in \overline{\varphi}$. Thus, $([x]\times[y]\times[e])\cap \varphi\neq\emptyset$. Since $[x]\subseteq \overline{X}$ and $[y]\subseteq \overline{Y}$, we obtain $(\overline{X}\times\overline{Y}\times [e])\cap \varphi\neq\emptyset$ and $e\in \oalpha(\varphi)_E(X,Y)$.

\smallskip

Fix $e\in \oalpha(\varphi)_E(X,Y)$. By definition, $(\overline{X}\times\overline{Y}\times [e])\cap \varphi\neq\emptyset$. So, there exist $x'\in \overline{X}$, $y'\in\overline{Y}$ and $e'\in [e]$ such that $\klam{x',y',e'}\in\varphi$. It follows that there exist $x\in X$ such that $[x]=[x']$ and $y\in Y$ such that $[y]=[y']$'. Since $[x]\times[y]\times[e]\cap \varphi\neq\emptyset$, the triple $\klam{x,y,e}\in\overline{\varphi}$. Thus, $(X\times Y\times\set{e})\cap\overline{\varphi}\neq\emptyset$ and, consequently, $e\in \poss{\overline{\varphi}}_E(X,Y)$.

\smallskip

\eqref{eq:poss-up-approx-2} Let $e\in \overline{\varphi}_E(X,Y)$. By definition, $(\overline{X}\times\overline{Y}\times [e])\cap \varphi\neq\emptyset$. So, there exist $x\in \overline{X}$, $y\in\overline{Y}$ and $e'\in [e]$ such that $\klam{x,y,e'}\in\varphi$. It follows that $\overline{X}\times\overline{Y}\times\set{e'}\cap \varphi\neq \emptyset$ and thus $e'\in \poss{\varphi}_E(\overline{X},\overline{Y})$. Since $[e]\cap \poss{\varphi}_E(\overline{X},\overline{Y})\neq \emptyset$, we obtain $e\in \overline{\poss{\varphi}_E(\overline{X},\overline{Y})}$.

\smallskip

Suppose that $e\in \overline{\poss{\varphi}_E(\overline{X},\overline{Y})}$. Then, $[e]\cap \poss{\varphi}_E(\overline{X},\overline{Y})\neq\emptyset$ and it follows that there exists $e'\in[e]\cap \poss{\varphi}_E(\overline{X},\overline{Y})$, i.e., $e'\in [e]$ and $(\overline{X}\times\overline{Y}\times\set{e'})\cap \varphi\neq\emptyset$. It is immediate to see that $(\overline{X}\times\overline{Y}\times [e])\cap \varphi\neq\emptyset$ and therefore $e\in \overline{\varphi}_E(X,Y)$.

\smallskip

\eqref{eq:suff-lo-approx-1} Assume that $e\in\suff{\uvarphi}_E(X,Y)$, i.e., $X\times Y\times\{e\}\subseteq\uvarphi$. We want to show that $\uX\times\uY\times[e]\subseteq\varphi$. To this end, let $\klam{x,y,f}\in \uX\times\uY\times[e]$. Thus $[x]\subseteq X$, $[y]\subseteq\uY$, and $[f]=[e]$. Therefore, $x\in X$ and $y\in Y$, and by the assumption we have  $\klam{x,y,e}\in\uvarphi$, which means that $[x]\times[y]\times[e]\subseteq\varphi$. In particular, $\klam{x,y,f}\in\varphi$, as desired.

\smallskip

\eqref{eq:suff-lo-approx-2} Let $e\in \ualpha_{E}(\varphi)(X,Y)$. Then, $\underline{X}\times\underline{Y}\times [e]\subseteq \varphi$. It follows that $\underline{X}\times\underline{Y}\times \set{e'}\subseteq \varphi$ for all $e'\in [e] $ and therefore $[e]\subseteq \suff{\varphi}_E(\underline{X},\underline{Y}) $, i.e., $e\in \underline{\suff{\varphi}_E(\underline{X},\underline{Y})} $. The reverse direction follows immediately by changing the direction the implications.
\end{proof}

\begin{example}
    The right-to-left inclusion of \eqref{eq:suff-lo-approx-1} fails, i.e., in general it does not have to be the case that
    \[
        \ualpha_{E}(\varphi)(X,Y)\subseteq\suff{\uvarphi}_{E}(X,Y)
    \]
    holds true. To show that it fails, we need to demonstrate that for some $X\subseteq A$, $Y\subseteq O$ and $e\in E$, $\uX\times\uY\times[e]\subseteq\varphi$, yet  $X\times Y\times\{e\}\nsubseteq\uvarphi$, i.e., there are $x\in X$ and $y\in Y$ such that $[x]\times[y]\times[e]\nsubseteq\varphi$. Again, let us use the ambient from Example \ref{ex:upper-approx-with-Robin}. If we consider $r\defeq\mathrm{Robin}$ and its equivalence class $[r]$ (the set of all pets similar to him), then for $X\defeq[r]\setminus\{r\}$, $\uX=\emptyset$. In consequence, we have  $\ualpha_{E}(\varphi)(X,Y)=E$. But, if we take $s$ to be a particular stick, and $e$ a particular dog park, and $\varphi$ an affordance of fetching, then it obviously does not have to be the case that $[r]\times[s]\times[e]\subseteq\varphi$, i.e., there may be a dog similar to Robin and a stick similar to $s$ such that the dog cannot fetch the stick in $e$ (in our case, e.g., due to certain hidden factors not accessible from the point of vie of information systems).\QED
\end{example}

From \eqref{eq:poss-up-approx-1} and \eqref{eq:poss-up-approx-2}, and from $\overline{\overline{X}}=\overline{X}$ we obtain that
\[
\poss{\overline{\varphi}}_{E}(\overline{X},\overline{Y})=\poss{\overline{\varphi}}_{E}(X,Y).
\]
\begin{proof}
\[
    \poss{\overline{\varphi}}_{E}(\overline{X},\overline{Y})=\overline{\poss{\varphi}_{E}\left(\overline{\overline{X}},\overline{\overline{Y}}\right)}=\overline{\poss{\varphi}_{E}(\overline{X},\overline{Y})}=\poss{\overline{\varphi}}_{E}(X,Y).\qedhere
\]
\end{proof}

Similarly, for the sufficiency operator, we have
\[
\suff{\uvarphi}_E(\uX,\uY)=\suff{\uvarphi}_E(X,Y).
\]
\begin{proof}
    ($\subseteq$) Using \eqref{eq:suff-lo-approx-1}, \eqref{eq:suff-lo-approx-2}, and $\uX=\underline{\uX}$ we obtain
    \[
\suff{\uvarphi}_E(\uX,\uY)\subseteq \underline{\suff{\varphi}_E(\underline{\uX},\underline{\uY})}=\underline{\suff{\varphi}_E(\uX,\uY)}=\suff{\uvarphi}_E(X,Y).
\]
\smallskip

($\supseteq$) In general, we have that $\uX\subseteq X$, and the sufficiency operator is antitone in both coordinates, so $\suff{\uvarphi}_E(X,Y)\subseteq \suff{\uvarphi}_E(\uX,\uY)$.
\end{proof}

\section{Conclusion and future work}

Above, we have introduced a mathematical model of affordances as ternary relations among actors, objects, and environments, and we have pointed to modal and approximation operators that constitute a~framework for reasoning with affordances. We have sketched the semantics, yet we shied away from logic as such, and in the next step, we aim to develop a~system of logic for the operators presented. Such a system requires, clearly, a multi-sorted approach, as we essentially work with three sorts of objects. The logic is going to be a mixed logic in the sense of \cite{Duntsch-et-al-AMLWBO}, as we employ the operator of sufficiency next to possibility. Producing its syntax and a fully fledged semantics is a natural continuation of the ruminations from this paper.

The above calls for addressing the following problem as well: are there ``natural'' axioms for affordances, for example, for knowledge spaces in learning theory?

So far, on the level of operators, we have only used the definition of an affordance in terms of information systems as a source of intuition, allowing us to talk about equivalence classes of similar objects from three different sorts. Therefore, another task is to develop tools that allow going down to the level of information systems and properties, and use them in reasoning with affordances.

Last, but not least, we aim to develop \emph{actions} as arising from affordances, and \emph{concepts} as emerging from actions. With the operators at hand, we may, for example, take a step similar to Wille's approach  \cite{wille82}, where (formal) concepts were defined by means of a sufficiency operator. This, however, is not the only possibility, as we would like to be less formal and try to capture concepts as reasonable models of concepts arising in ecological circumstances.

\section*{Acknowledgments}

\begin{sloppypar}
This research was funded by the National Science Center (Poland), grant number~2020/39/B/HS1/00216. Ivo D\"untsch  would like to thank Nicolaus Copernicus University in Toru\'n for supporting him through the ``Excellence Initiative--Research University'' program.

The authors would like to thank the anonymous reviewers for their insightful and fair comments, which helped improve the paper.
\end{sloppypar}

\bibliographystyle{apalike}
\providecommand{\noop}[1]{}

\end{document}